\documentclass[12pt,twoside,reqno]{amsart}
\usepackage[colorlinks=true,citecolor=blue]{hyperref}
\usepackage{mathptmx, amsmath, amssymb, amsfonts, amsthm, mathptmx, enumerate, color,mathrsfs}
\usepackage{graphicx}

\usepackage{multirow}
\usepackage{epstopdf}
\usepackage{multicol}
\usepackage{algorithm}
\usepackage{algorithmic}
\usepackage{epstopdf}

\newtheorem{theorem}{Theorem}[section]
\newtheorem{lemma}[theorem]{Lemma}
\newtheorem{proposition}[theorem]{Proposition}
\newtheorem{corollary}[theorem]{Corollary}

\theoremstyle{definition}
\newtheorem{definition}[theorem]{Definition}

\newtheorem{example}[theorem]{Example}

\newtheorem{remark}[theorem]{Remark}
\numberwithin{equation}{section}

\newcommand{\Pb}{\mathbb P}
\newcommand{\R}{\mathbb R}
\newcommand{\N}{\mathbb N}

\newcommand{\X}{\mathbb X}
\newcommand{\Y}{\mathbb Y}
\newcommand{\Uball}{{\mathbb B}}

\newcommand{\dom}{{\rm dom}\, }
\newcommand{\graph}{{\rm gph}\,}
\newcommand{\epi}{{\rm epi}\,}

\newcommand{\nullv}{\mathbf{0}}

\newcommand{\conv}{{\rm co}\, }
\newcommand{\clco}{{\rm clco}\, }

\newcommand{\cone}{{\rm cone}\, }

\newcommand{\inte}{{\rm int}\, }

\newcommand{\Gdelta}{{\rm G}_\delta}       

\newcommand{\OP}{{\sf P}}    
\newcommand{\POP}{{\sf P}_p}    

\newcommand{\val}{{\sf val}}    
\newcommand{\sel}{\textsf{s}}     
\newcommand{\FReg}{{\sf R}}    
\newcommand{\Argmin}{{\sf Argmin}}    

\newcommand{\eva}{{\rm v}}     

\newcommand{\Lin}{{\mathcal L}}   
\newcommand{\SO}{{\mathbf{SO}}}   

\newcommand{\sur}[1]{{\rm sur}(#1)}    

\newcommand{\POPp}[1]{{\sf P}_{#1}}

\newcommand{\ball}[2]{{\rm B}\left[#1;#2\right]}   

\newcommand{\dist}[2]{{\rm dist}\left(#1;#2\right)}
\newcommand{\exc}[2]{{\rm exc}(#1;#2)}
\newcommand{\incr}[2]{{\rm inc}(#1;#2)}    
\newcommand{\Ncone}[2]{{\rm N}(#1;#2)}    
\newcommand{\haus}[2]{{\rm haus}(#1;#2)}

\newcommand{\Coder}[2]{{\rm D}^*#1(#2)}
\newcommand{\ind}[2]{\iota(#1;#2)}

\begin{document}
\setcounter{page}{1}

\vspace*{2.0cm}
\title[Marginal Analysis in Optimization with Set-Valued Inclusion constraints]
{Marginal Analysis of Convex Optimization Problems with Set-Valued Inclusion Constraints}
\author[A. Uderzo]{ Amos Uderzo$^{1,*}$ }
\maketitle
\vspace*{-0.6cm}

\begin{center}
{\footnotesize

$^1$Department of Mathematics and Applications, University of Milano-Bicocca, Milano (Italy) \\

}\end{center}

\vskip 4mm {\footnotesize \noindent {\bf Abstract.}
In this paper, stability and sensitivity properties of a class of parametric constrained
optimization problem, whose feasible region is defined
by a set-valued inclusion, are investigated through the associated
optimal value function. Set-valued inclusions are a kind of constraint system,
which naturally emerges in contexts requiring the robust fulfilment
of traditional cone constraints, where data are affected by uncertain elements
having a non stochastic nature, or in (MPEC) as a vector equilibrium constraint,
where feasible solutions are intended as equilibrium point in a strong sense.
Under proper convexity assumptions on the objective function and the
constraining set-valued term, combined with a global qualification
condition, a class of parametric optimization problems is singled out, which
displays a global Lipschitz behaviour. By employing recent results of
variational analysis, elements for a sensitivity analysis of this
class of problems are provided via exact subgradient formulae for the optimal
value function. Further consequences of the stability behaviour are
explored in terms of problem calmness and viability of penalization
techniques.

 \noindent {\bf Keywords.}
Set-valued inclusions; $C$-concave multi-valued mappings;
metric $C$-increase; parametric optimization;  optimal value function;
problem calmness; penalty function.

\noindent {\bf 2020 Mathematics Subject Classification.}
90C30, 90C31, 90C48. }

\renewcommand{\thefootnote}{}
\footnotetext{ $^*$Corresponding author.
\par
E-mail address: amos.uderzo@unimib.it (A. Uderzo).
\par

}

\vskip2cm

\section{Introduction}

The present paper aims at bringing new insights into the study
of the stability and sensitivity properties of parametric optimization
problems with set-valued inclusion constraints. Stability and sensitivity
issues for constrained optimization problems subject to parameter
perturbations lie at the core of a well-recognized and very active
area of variational analysis, often indicated as perturbation
analysis of optimization problems (see \cite{BaGuKlKuTa83,BonSha00,Fiac83,LuMiSa02,Mord18}).
According to a deep-rooted approach to this topic,
the two main objects of study are the \textit{optimal value}
(a.k.a. \textit{marginal} or \textit{performance}) \textit{function}
associated with a class of parametric optimization problems
and the (generally) multi-valued mapping collecting the
optimal solutions (if any) of these problems, when the value of
the parameter varies. Various qualitative and quantitative information
describing how problems are conditioned by and react to changes
in the parameter (e.g. solvability near reference values,
closeness to known values, rate of changes and so on)
can be grasped by analyzing these two objects. The investigations
reported in the present paper focus on the analysis of the
optimal value function, whereas the analysis of the solution
mapping will be left for a future, specifically devoted, project of
research.
According to \cite[Chapter 4.6]{Mord18}
\begin{quote}
{\it it would not be
an exaggeration to say that marginal functions manifest the essence
of modern techniques in variational analysis involving perturbation
and approximation procedures with the subsequent passing to the limit.}
\end{quote}

The distinguishing feature of the optimization problems considered in
the present paper consists in the fact that their constraints are
expressed by parameterized set-valued inclusions. These lead to a constraint system
format different from the generalized equation format typically occurring
in the variational analysis literature, with its own geometry
needing specifically devised tools to be adequately handled.
Historically, to the best of the author's knowledge, the appearance
of such a constraint format can be traced back to \cite{Soys72},
where in order to address inexact linear programming problems feasible
regions were considered, which are defined via "set-containment", so
constraints are called there "set-inclusive". Subsequently, set-valued
inclusions revealed to be a natural language to formalize the fulfilment
of traditional cone constraints in the context of robust optimization,
as it was introduced by A. Ben-Tal and A. Nemirovsky (see \cite{BenNem98,BeGhNe09}).
Besides, they found further contexts of relevant application in vector
optimization, in expressing ideal efficiency (see \cite{Uder24}), and,
more in general, in vector equilibrium theory, in formalizing the concept
of strong solution to vector Ky-Fan inequalities (see \cite{Uder23}).

Let us consider the following class of parametric constrained optimization
problems
\begin{equation*}
  (\POP) \hskip 1.8cm \begin{array}{cl}
    \displaystyle\min_{x\in\X}  & \varphi(p,x) \\
      \hbox{ sub } & F(p,x)\subseteq C,
  \end{array}
\end{equation*}
where $\varphi:\Pb\times\X\longrightarrow\R$ represents the objective
function, $\{\nullv\}\subsetneqq C\subsetneqq\Y$ is a (nontrivial) closed, convex
cone, and $F:\Pb\times\X\rightrightarrows\Y$ is a multi-valued mapping
defining the set-valued inclusion problem, which formalizes the constraint
system of the optimization problems.
Throughout the paper, $(\Pb,\|\cdot\|)$, $(\X,\|\cdot\|)$, and $(\Y,\|\cdot\|)$
denote Banach spaces over the real field $\R$, with null vector $\nullv$.
As should be clear, the variable
$p\in\Pb$ indicates the parameter subject to perturbation, while the variable
$x\in\X$ stands for the problem unknown.
In the above setting, the (generally) multi-valued mapping $\FReg:\Pb\rightrightarrows\X$
that describes the changes of the feasible region of $(\POP)$, when the parameter $p$ varies,
is given by
\begin{equation}
  \FReg(p)=F^{+1}(p,\cdot)(C)=\{x\in X\ |\ F(p,x)\subseteq C\},
\end{equation}
where the notation $F^{+1}$ (borrowed from \cite{AubFra90}) indicates
one of the two possible manners to define the inverse image through the set-valued
mapping $F$ of a given subset of the range space, namely the core
through $F$ of that subset.
The optimal value function $\val:\Pb\longrightarrow\R\cup\{\pm\infty\}$
associated with the class of parametric constrained optimization problems
$(\POP)$ takes the form
\begin{equation}
  \val(p)=\inf_{x\in\FReg(p)}\varphi(p,x),
\end{equation}
while the problem solution set-valued mapping $\Argmin:\Pb\rightrightarrows\X$
is given by
\begin{equation*}
  \Argmin(p)=\{x\in\FReg(p)\ |\ \varphi(p,x)=\val(p)\}.
\end{equation*}
The investigations exposed in the present paper carry on the value analysis
started in \cite{Uder21}, where several calmness properties of the marginal function
$\val$ have been established by employing various sufficient conditions for the local
Lipschitz semicontinuity behaviour of $\FReg$.
Although conducted within the same framework, the present analysis aims nonetheless
at enlightening a different aspect of this topic. The idea triggering the investigations here contained
relies on the expectation that much stronger global properties of $\val$ can be
obtained under suitable convexity assumptions on the problem data, which
in \cite{Uder21} are completely disregarded. As one expects, convexity
alone is not able to ensure those properties of $\val$, so its effects
are studied in synergy with qualification conditions for set-valued inclusions,
which are expressed in terms of $C$-increase behaviour of the set-valued term $F$.
More precisely, it can be shown that convexity and Lipschitz continuity of $\val$
can be achieved by applying a recently established result about the
existence of continuous selection of $\FReg$. This leads to single out
a class of problems $(\POP)$, that will be called "qualified convex optimization
problems with set-valued inclusion constraints", for which either $\val\equiv-\infty$
or $\val$ is convex and locally Lipschitz. The reader should notice that,
in a finite-dimensional setting, this means that $\val$ turns out to be
smooth on $\Pb$, up to a (Lebesgue) residual subset.
Furthermore, elements for the sensitivity analysis of $(\POP)$
are proposed by providing sharp estimates of the subgradient
of $\val$, which are established by means of calculus rules for
the Moreau-Rockafellar subdifferential and coderivative calculus
for multifunctions with convex graph.

The contents of the paper are arranged according to the following
scheme.
Section \ref{Sect:2} collects all the technical tools, mainly coming
from convex analysis and from the theory of set-valued inclusions, needed
in the subsequent analysis.
Section \ref{Sect:3} exposes the main results of the paper, namely Lipschitzian
properties of $\val$ and exact formulae for its subgradients, then illustrates
them through several examples.
In Section \ref{Sect:4} some remarkable consequences of the results
presented in the previous section are discussed, which deal with the
property of problem calmness and the viability of penalization techniques
for solving problems $(\POP)$.

The notations in use throughout the paper are standard. Whenever $(\X,\|\cdot\|)$
is a Banach space, $\X^*$ denotes its topological dual, with $\langle\cdot,\cdot\rangle$
indicating their duality pairing. Given $x\in\X$ and $r\ge 0$, the
closed ball centered at $x$, with radius $r$ is denoted by $\ball{x}{r}$.
In particular, $\ball{\nullv}{1}=\Uball$, while the unit ball in a dual
space is denoted by $\Uball^*$.

Given a subset $S$ of a Banach space, $\inte S$ denotes the topological interior of $S$,
whereas $\cone S$ and $\clco S$ denote the conical hull and the convex closure
of $S$, respectively. Given an element $x$ in the same space, by $\dist{x}{S}=\inf_{z\in S}
\|x-z\|$ the distance of $x$ from $S$ is indicated.
Given two subsets $A$ and $B$ of the same Banach space, the excess of $A$ beyond
$B$ is denoted by $\exc{A}{B}=\sup_{a\in A}\dist{a}{B}$, with $\haus{A}{B}=\max
\{\exc{A}{B},\, \exc{B}{A}\}$ denoting the Pompeiu-Hausdorff distance of $A$ and $B$.
$\Lin(\X,\Y)$ stands for the Banach space of all linear bounded operators
from $\X$ to $\Y$.
The acronyms l.s.c. and u.s.c. stand for lower semicontinuous and upper
semicontinuous, respectively.

Throughout the paper the following standing assumptions are maintained:

\begin{itemize}
  \item[($\textsf{a}_1$)] $F$ takes nonempty and closed values (in particular,
  $\dom F=\Pb\times\X$);

  \item[($\textsf{a}_2$)] $\dom\varphi=\Pb\times\X$;
\end{itemize}

Moreover, whenever considered, the space $\R^n$ is assumed to be equipped with its
usual Euclidean space structure.


\section{Preliminaries}    \label{Sect:2}

\subsection{Elements of convex analysis}

Whenever $\Omega\subseteq\X$ is a nonempty convex subset of
a normed space and $\bar x\in\Omega$, the set
$$
  \Ncone{\Omega}{\bar x}=\{x^*\in\X^*\ |\ \langle x^*,x-\bar x\rangle
  \le0,\quad\forall x\in\Omega\}
$$
is called the {\it normal cone} ({\it in the sense of convex analysis})
to $\Omega$ at $\bar x$.

Given a convex function $\psi:\X\longrightarrow\cup\{\pm\infty\}$
and $\bar x\in\dom\psi$, the {\it subdifferential in the sense of
convex analysis} (a.k.a. {\it Moreau-Rockafellar subdifferential})
of $\psi$ at $\bar x$ is defined by
$$
   \partial\psi(\bar x)=\{x^*\in\X^*\ |\ \langle x^*,x-\bar x\rangle
   \le\psi(x)-\psi(\bar x),\quad\forall x\in\X\}.
$$
It is well known that if $\Omega$ is a closed and convex set, then
the distance function $x\mapsto\dist{x}{\Omega}$ from $\Omega$ is a
(Lipschitz continuous) convex function. In this special case, the two above
dual constructions are linked by the following relationships, which
will be useful in the sequel.

\begin{proposition}[\cite{MorNam22}]    \label{pro:subdnconerel}
Let $\Omega$ be a closed and convex set and let $\bar x\in\Omega$.
The following equalities hold:
\begin{itemize}
  \item[(i)] $\partial\dist{\cdot}{\Omega}(\bar x)=\Ncone{\Omega}{\bar x}\cap\Uball^*$;

  \item[(ii)] $\Ncone{\Omega}{\bar x}=\bigcup_{t\ge 0} t\partial\dist{\cdot}{\Omega}(\bar x)$.
\end{itemize}
\end{proposition}

Their are also linked by the indicator function $x\mapsto\ind{x}{\Omega}$, which
is convex iff $\Omega\subseteq\X$ is so, through the formula
\begin{equation}     \label{eq:indsubd}
  \partial\ind{\cdot}{\Omega}(\bar x)=\Ncone{\Omega}{\bar x},
\end{equation}
with $\bar x\in\Omega$ (see, for instance, \cite[Example 3.55]{MorNam22}).

Among the various and well developed calculus rules for subdifferential,
the one concerning the composition with linear mappings, recalled below,
will be of use in the next section: let $\Lambda\in\Lin(\X,\Y)$, let $\bar x\in\X$
and let $\psi:\Y\longrightarrow\R\cup\{\pm\infty\}$ be continuous at $\Lambda\bar x$.
Then, it holds
\begin{equation}  \label{eq:subdlincomp}
   \partial\left(\psi\circ\Lambda\right)(\bar x)=\Lambda^*
   \partial\psi(\Lambda\bar x).
\end{equation}
A proof of the above composition rule can be found, for instance,
in \cite[Theorem 3.55]{MorNam22}.

The further tool needed for the subsequent analysis is a kind of
generalized derivative for convex set-valued mappings, i.e. mappings
whose graph is convex, which can be built graphically via the normal
cone. Let $G:\X\rightrightarrows\Y$ be a convex set-valued
mapping between normed spaces and let $(\bar x,\bar y)\in\graph G$.
The {\it coderivative} of $G$ at $(\bar x,\bar y)$ ({\it in the sense of
convex analysis}) is the set-valued mapping $\Coder{G}{\bar x,\bar y}:\Y^*
\rightrightarrows\X^*$
$$
  \Coder{G}{\bar x,\bar y}(y^*)=\{x^*\in\X^*\ |\ (x^*,-y^*)\in
  \Ncone{\graph G}{(\bar x,\bar y)}\}.
$$
Several advanced calculus rules in convex analysis can be performed
provided that proper qualification conditions are satisfied. Fur the purposes of the
present analysis, the following is needed: a given pair of nonempty
subsets $A$ and $B$ of a normed space is said to be {\it subtransversal}
at $\bar x\in A\cap B$ if there exist positive $\kappa$ and $r$
such that
$$
  \dist{x}{A\cap B}\le\kappa\max\{\dist{x}{A},
  \dist{x}{B}\},\quad\forall x\in\ball{\bar x}{r}.
$$
Such a condition, which generalizes the classical concept of transversality
in mathematical analysis and differential topology in capturing a certain
"good mutual arrangement" of several sets in space, appeared under
different names in different contexts of variational analysis.
It is worth observing that, whenever $A$ and $B$ are closed convex sets,
a sufficient condition for the pair $A$ and $B$ to be subtransversal at
$\bar x\in A\cap B$ is that $\nullv\in\inte(A-B)$. Pairs of polyhedral
convex sets are known to be subtransversal at any intersection point.
For a systematic study of this properties the reader is referred to
\cite{BauBor96,KrLuTh17} and references therein.

\vskip.5cm

\subsection{$C$-Increasing multi-valued mappings}

The property recalled below is employed to guarantee solvability of set-valued
inclusion problems and related error bounds, thereby propitiating stability behaviours
of their solution mappings.

\begin{definition}
Let $G:\X\rightrightarrows\Y$ be a set-valued mapping between vector normed
spaces and let $C\subset\Y$ be a closed convex cone. $G$ is said to be {\it
metrically $C$-increasing} at $x_0\in\X$ if there exist $\alpha>1$ and $\delta>0$
such that
\begin{equation}   \label{in:mCincrdef}
  \forall r\in (0,\delta]\ \exists u\in\ball{x_0}{r}\ |\
  \ball{G(u)}{\alpha r}\subseteq\ball{G(x_0)+C}{r}.
\end{equation}
The value
$$
  \incr{G}{x_0}=\sup\{\alpha>1\ |\ \exists\delta>0
  \hbox{ for which (\ref{in:mCincrdef}) holds}\}
$$
is called {\it exact bound of $C$-increase} of $G$ at $x_0$.
\end{definition}

Connections of the metric $C$-increase property and the decrease
principle (in the sense of Borwein-Clarke-Ledyaev) as well as
with Caristi type conditions are discussed in \cite{Uder22}
and \cite{Uder24}, respectively. Some examples of classes
of metrically $C$-increasing mappings are provided below.

\begin{example}[Rescaled rotations with Lipschitz additive
perturbations]
For any $n\ge 2$, let $\X=\Y=\R^n$ and $C=\R^n_+$.
Denote by $(\SO(n),\circ)$ the special orthogonal group,
consisting of all rotations acting on $\R^n$.
It has been proven that any rescaled rotation
$\lambda O\in\Lin(\R^n,\R^n)$, defined by $\lambda>n$ and $O\in
\SO(n)$, is $\R^n_+$-increasing at each point $x\in\R^n$, with
$$
  \incr{\lambda O}{x}\ge\sqrt{n},\quad\forall x\in\R^n
$$
(see, for more details, \cite[Example 2.2]{Uder24}).
By recalling that the $C$-increase property is preserved under small
additive Lipschitz perturbations (see \cite[Proposition 2.4]{Uder24}),
it is possible to see that, if $H:\R^n\rightrightarrows\R^n$ is
a Lipschitz continuous set-valued mapping, with constant $\ell$ such
that
$$
  \ell<1-\frac{1}{\sqrt{n}},
$$
then the set-valued mapping $\lambda O+H$ turns out to be
$\R^n_+$-increasing at each point $x\in\R^n$ with
$$
  \incr{\lambda O+H}{x}\ge (1-\ell)\sqrt{n}.
$$
\end{example}

The next example points out connections of the metric $C$-increase
property with openness (a.k.a. covering) at a linear rate, whose phenomenology
has been the subject of intensive investigations in variational analysis.

\begin{example}[Compactly generated fans]    \label{ex:compgenfan}
Given a nonempty, convex and compact subset of linear mappings
${\mathcal G}\subset\Lin(\R^n,\R^m)$, the set-valued mapping
$H_{\mathcal G}:\R^n\rightrightarrows\R^m$ defined as
$$
  H_{\mathcal G}(x)=\{\Lambda x\ |\ \Lambda\in {\mathcal G}\}
$$
is known to fulfil the following properties:
\begin{itemize}
  \item[(i)]  $H_{\mathcal G}(t x)=t H_{\mathcal G}(x),
  \quad\forall t>0,\ \forall x\in\R^n$;

  \item[(ii)] $H_{\mathcal G}(x)$ is nonempty, closed and convex for every $x\in\R^n$;

  \item[(iii)] $H_{\mathcal G}(x_1+x_2)\subseteq H_{\mathcal G}(x_1)+
  H_{\mathcal G}(x_2),\quad\forall x_1,\, x_2\in\R^n$.
\end{itemize}
Therefore $H_{\mathcal G}$ falls in the class of those set-valued mappings
which are called {\it fans} after \cite{Ioff81}.
Let $C\subseteq\R^m$ be a closed, convex and pointed cone.
In \cite[Proposition 2.15 + Corollary 2.18]{Uder22} it has been proven that
if
$$
  \inf_{\Lambda\in{\mathcal G}}\sur{\Lambda}=\eta_{\mathcal G}>0
$$
and
$$
  \inte\left(\bigcap_{\Lambda\in{\mathcal G}}\Lambda^{-1}(C)\right)
  \ne\varnothing,
$$
then $H_{\mathcal G}$ is metrically $C$-increasing at each point
$x\in\R^n$ with $\incr{H_{\mathcal G}}{x}\ge\eta_{\mathcal G}+1$.
Here $\sur{\Lambda}$ denotes the {\it exact bound of open covering}
of the linear mapping $\Lambda$, which provides a sharp estimate
of its surjectivity behaviour as follows
$$
  \sur{\Lambda}=\sup\{\eta>0\ |\ \Lambda\Uball\supseteq\eta\Uball\}.
$$
Such kind of bound was fully investigated at the early stage of the
theory of metric regularity and the following characterization in terms
of the Banach constant of the adjoint mapping $\Lambda^*$ to $\Lambda$
was soon understood to hold true
$$
  \sur{\Lambda}=\dist{\nullv}{\Lambda^*\Uball^*}=\inf_{\|y^*\|=1}
  \|\Lambda^*y^*\|
$$
(see, for instance, \cite[Corollary 1.58]{Mord06}).
The reader should also notice that any set-valued mapping such
as $H_{\mathcal G}$ takes compact values and turns out to be Lipschitz
continuous with respect to the Pompeiu-Hausdorff distance, i.e. there
exists $\ell\ge 0$ such that
$$
  \haus{H_{\mathcal G}(x_1)}{H_{\mathcal G}(x_2)}\le\ell\|x_1-x_2\|,
  \quad\forall x_1,\, x_2\in\X
$$
(see, for more details, \cite[Remark 2.14(iii)]{Uder22}).
\end{example}

\vskip.5cm


\subsection{$C$-concave multi-valued mappings}

Set-valued mappings may exhibit different forms of convexity property, the
mostly employed in variational analysis being the convexity with reference
to the graph. Nevertheless, in connection with the theory of set-valued
inclusions, the following concept of concavity (which can not be of graphical
nature) reveals to be useful.

\begin{definition}[$C$-concavity]
Let $C\subseteq\Y$ be a convex cone. A set-valued mapping $G:\X\rightrightarrows\Y$
is said to be $C$-concave on $A\subseteq\X$ is for every $x_1,\, x_2\in A$ and
$t\in [0,1]$ it holds
$$
  G(tx_1+(1-t)x_2)\subseteq tG(x_1)+(1-t)G(x_2)+C.
$$
\end{definition}

\begin{example}    \label{ex:Cconcgomega}
Let $C\subseteq\Y$ be a convex cone. Let us recall that,
following \cite{FreKas99}, a single-valued mapping $g:\X\longrightarrow\Y$
between vector spaces is said to be $C$-concave if
$$
   g(tx_1+(1-t)x_2)\in tg(x_1)+(1-t)g(x_2)+C,
   \quad\forall x_1,\, x_2\in\X, \ \forall t\in [0,1].
$$
For instance, it is well known that, whenever $\Y=\R^m$, $C=\R^m_+$,
and $g=(g_1,\dots,g_m)$ is defined by scalar functions $g_i:\X\longrightarrow\R$
which are concave for every $i=1,\dots,m$, then $g$ turns out to be
$\R^m_+$-concave.

Now, assume that $\Omega$ is a nonempty set and $g:\X\times\Omega\longrightarrow\Y$
is a given mapping such that $g(\cdot,\omega):\X\longrightarrow\Y$  is
$C$-concave for every $\omega\in\Omega$. Then, according to
\cite[Example 4.4]{Uder21}, the set-valued mapping $G_{g,\Omega}:\X\rightrightarrows\Y$
defined by
$$
   G_{g,\Omega}(x)=g(x,\Omega)=\{g(x,\omega)\ |\ \omega\in\Omega\}
$$
is $C$-concave.
In particular, if taking $\Omega=C$ and $g:\X\times C\longrightarrow\Y$
given by
$$
   g(x,y)=h(x)+y,
$$
where $h:\X\longrightarrow\Y$ is $C$-concave, then
the resulting set-valued mapping $G_{h,C}=h+C$ is $C$-concave.
It is also worth noticing that, if taking $\Omega={\mathcal G}
\subseteq\Lin(\R^n,\R^m)$ as in Example \ref{ex:compgenfan} and
$g:\R^n\times {\mathcal G}\longrightarrow\R^m$ given by
$$
  g(x,\Lambda)=\Lambda x,
$$
by linearity of $\Lambda$ one obtains that any compactly generated
fan is $C$-concave with respect to any cone $C\subseteq\R^m$.
\end{example}

\begin{example}[Separable fans]
Let $\mathcal{A}=\{A_1,\dots,A_n\}$, where $A_i\subseteq\R^n$
is a nonempty, convex and compact set for every $i=1,\dots,n$.
Then, consider the set-valued mapping $H_\mathcal{A}:\R^n\rightrightarrows\R^n$
defined as being
$$
  H_\mathcal{A}(x)=\sum_{i=1}^{n}A_ix_i.
$$
It is readily seen that $H_\mathcal{A}$ takes nonempty, convex compact
values and it holds
$$
  H_\mathcal{A}(tx)=\sum_{i=1}^{n}A_i tx_i=t H_\mathcal{A}(x),\quad
  \forall t>0,\, x\in\R^n,
$$
and
\begin{eqnarray*}
  H_\mathcal{A}(x+z)&=&\sum_{i=1}^{n}A_i(x_i+z_i)\subseteq
  \sum_{i=1}^{n}(A_ix_i+A_iz_i)=\sum_{i=1}^{n}A_ix_i+\sum_{i=1}^{n}A_iz_i  \\
  &=& H_\mathcal{A}(x)+H_\mathcal{A}(z),\quad\forall x,\, z\in\R^n.
\end{eqnarray*}
This means that $H_\mathcal{A}$ is a fan, so it is $C$-concave with respect
to any convex cone $C\subseteq\R^n$.

To the best of the author's knowledge, such kind of multi-valued mappings
was introduced in \cite{Soys72} in the context of convex optimization,
when the terminology of "fan" was not yet wide-spread in the related literature.
\end{example}

The next proposition makes clear the role of $C$-concavity in ensuring
convexity properties for the solution mapping associated with set-valued
inclusion problems.

\begin{proposition}   \label{pro:FRegconv}
With reference to the constraint system of problems ($\POP$), if
$F:\Pb\times\X\rightrightarrows\Y$ is $C$-concave, then $\FReg:\Pb
\rightrightarrows\X$ is a convex set-valued mapping.
\end{proposition}

\begin{proof}
Let $p_1,\, p_2\in\dom\FReg$ and $t\in [0,1]$. For arbitrary $x_1\in\FReg(p_1)$
and $x_2\in\FReg(p_2)$ it holds
$$
  F(p_1,x_1)\subseteq C \quad \hbox{ and }\quad  F(p_2,x_2)\subseteq C,
$$
whence, by $C$-concavity of $F$ it follows
\begin{eqnarray*}
  F(tp_1+(1-t)p_2,tx_1+(1-t)x_2) &=& F(t(p_1,x_1)+(1-t)(p_2,x_2)) \\
   &\subseteq & tF(p_1,x_1)+(1-t)F(p_2,x_2)+C \\
   &\subseteq & tC+(1-t)C+C=C.
\end{eqnarray*}
By arbitrariness of $x_1\in\FReg(p_1)$ and $x_2\in\FReg(p_2)$, the above
inclusion shows that
$$
  t\FReg(p_1)+(1-t)\FReg(p_2)\subseteq\FReg(tp_1+(1-t)p_2).
$$
If $p_1$ or $ p_2$ is not in $\dom\FReg$, the convention $S+\varnothing=\varnothing$
for every set $S$ makes the last inclusion still valid.
\end{proof}

In what follows, a subset $S\subseteq\Y$ is said to be $C$-bounded
if the set $S\backslash C$ is (metrically) bounded.
The combination of global metric $C$-increase and $C$-concavity of $F$,
along with some technical assumptions,
yields global solvability and continuous parameter dependence for
the solutions to parameterized set-valued inclusion problems.
This fact is stated in the next result, whose formulation requires to
define the following problem constant
\begin{equation}\label{eq:defmCincreb}
  \alpha_F=\inf\left\{\incr{F(p,\cdot)}{x}\ |\ (p,x)\in\Pb\times\X,\
  F(p,x)\nsubseteq C\right\}.
\end{equation}

\begin{theorem}[Continuous selection and global error bound]    \label{thm:cselectFreg}
With reference to the constraint system of problems ($\POP$), suppose that:
\begin{itemize}
  \item[(i)] $\forall p\in\Pb$ $\exists\hat x_p\in\X$ such that
  $F(p,\hat x_p)$ is $C$-bounded;

  \item[(ii)] $F(p,\cdot):\X\rightrightarrows\Y$ is l.s.c. on $\X$, $\forall p\in\Pb$ ;

  \item[(iii)]$F(p,\cdot):\X\rightrightarrows\Y$ is $C$-concave on $\X$, $\forall p\in\Pb$ ;

  \item[(iv)] $F(p,\cdot):\X\rightrightarrows\Y$ is Hausdorff $C$-u.s.c.
  on $\Pb$, $\forall x\in\X$;

  \item[(v)] it holds $\alpha_F>1$.
\end{itemize}
Then, $\dom\FReg=\Pb$ and $\FReg:\Pb\rightrightarrows\X$ admits a continuous selection.
Moreover, the following estimate holds for any $\alpha\in (1,\alpha_F)$
\begin{equation}    \label{in:glerbosvi}
  \dist{x}{\FReg(p)}\le\frac{\exc{F(p,x)}{C}}{\alpha-1},\quad
  \forall (p,x)\in\Pb\times\X.
\end{equation}
\end{theorem}

\begin{proof}
It suffices to observe that,
as a metric space, $(\Pb,\|\cdot\|)$ is a paracompact topological space,
and then to apply \cite[Theorem 3.4]{Uder24}.
The estimate in (\ref{in:glerbosvi}) is actually valid under more general assumptions
by virtue of \cite[Proposition 3.1]{Uder24}.
\end{proof}

\vskip1cm


\section{Stability and sensitivity results}     \label{Sect:3}

By exploiting the analysis tools recalled in the previous section,
it is possible to establish several properties of $\val$ of both qualitative
and quantitative interest in studying stability and sensitivity issues
with reference to ($\POP$).
Let us start with the convexity behaviour of $\val$, which is a straightforward
consequence of a well-known phenomenon in parametric constrained optimization,
occurring whenever a convex objective function is to be minimized
over a parameter dependent feasible region, which is a convex multifunction
of the parameter (see, for instance, \cite[Theorem 2.129]{MorNam22}).
The proof is presented here in full details for the sake of completeness.

\begin{proposition}[Convexity of $\val$]    \label{pro:convval}
Given a family of problems ($\POP$), suppose that:
\begin{itemize}
  \item[(i)] $\varphi:\Pb\times\X\longrightarrow\R$ is convex;

  \item[(ii)] $F:\Pb\times\X\rightrightarrows\Y$ is $C$-concave.
\end{itemize}
Then, $\val:\Pb\longrightarrow\R\cup\{\pm\infty\}$ is convex.
\end{proposition}

\begin{proof}
Take arbitrary $p_1,\, p_2\in\Pb$ and $t\in [0,1]$ and suppose first that
$tp_1+(1-t)p_2\in\dom\val$.
Then, it must be $\FReg(tp_1+(1-t)p_2)\ne\varnothing$. By combining the convexity
of the set-valued mapping $\FReg$ established in Proposition \ref{pro:FRegconv}
with the convexity of $\varphi$, one finds
\begin{eqnarray*}
  \val(tp_1+(1-t)p_2) &=& \inf_{x\in\FReg(tp_1+(1-t)p_2)} \varphi(tp_1+(1-t)p_2,x) \\
   &\le & \inf_{x\in\, [t\FReg(p_1)+(1-t)\FReg(p_2)]} \varphi(tp_1+(1-t)p_2,x) \\
   &=& \inf_{x_1\in\FReg(p_1) \atop x_2\in\FReg(p_2)} \varphi(tp_1+(1-t)p_2,tx_1+(1-t)x_2) \\
   &\le&  \inf_{x_1\in\FReg(p_1) \atop x_2\in\FReg(p_2)} [t\varphi(p_1,x_1)+(1-t)\varphi(p_2,x_2)] \\
   &=& t\val(p_1)+(1-t)\val(p_2).
\end{eqnarray*}
In the case $\val(tp_1+(1-t)p_2)=+\infty$, by virtue of ($\textsf{a}_2$) it must
$\FReg(tp_1+(1-t)p_2)=\varnothing$. This fact by convexity of $\FReg$ implies that $\FReg(p_1)=\FReg(p_2)=
\varnothing$ and hence $\val(p_1)=\val(p_2)=+\infty$, so one obtains
$$
  \val(tp_1+(1-t)p_2)=+\infty=t\val(p_1)+(1-t)\val(p_2).
$$
In the last case, in which $\val(tp_1+(1-t)p_2)=-\infty$, nothing is left to prove, so
the proof is complete.
\end{proof}

Under proper qualification conditions on the set-valued inclusion, which formalizes the constraint
system of problems ($\POP$), the convexity of $\val$ entails the local Lipschitz continuity property,
which is much stronger than mere continuity appearing in Berge's type theorems (see \cite{AliBor06}),
as well as than those forms of calmness established in \cite{Uder21}.

\begin{theorem}[Local Lipschitz continuity of $\val$]    \label{thm:lLipval}
With reference to the family of problems ($\POP$), suppose that:
\begin{itemize}
 \item[(i)] $\varphi:\Pb\times\X\longrightarrow\R$ is convex and continuous;

  \item[(ii)] $\exists\hat x_p\in\X:\ $ $F(p,\hat x_p)$ is $C$-bounded, $\forall p\in\Pb$ ;

  \item[(iii)] $F(p,\cdot):\X\rightrightarrows\Y$ is l.s.c. on $\X$, $\forall p\in\Pb$;

  \item[(iv)] $F:\Pb\times\X\rightrightarrows\Y$ is $C$-concave;

  \item[(v)] $F(\cdot,x):\Pb\rightrightarrows\Y$ is Hausdorff $C$-u.s.c. on $\Pb$, $\forall x\in\X$;

  \item[(vi)] it holds $\alpha_F>1$.
\end{itemize}
Then, either $\val=-\infty$ or $\dom\val=\Pb$ and $\val$ is locally
Lipschitz around each point $p\in\Pb$.
\end{theorem}

\begin{proof}
By Proposition \ref{pro:convval}, $\val:\Pb\longrightarrow\R\cup\{\pm\infty\}$
is convex. Notice that, as a consequence of assumption (iv), each set-valued mapping
$F(p,\cdot)$ is $C$-concave, for every $p\in\Pb$.
Thus, all hypotheses being satisfied, Theorem \ref{thm:cselectFreg}
applies, so $\dom\FReg=\Pb$  and $\val(p)<+\infty$ for every $p\in\Pb$.
Now, if there exists $p_0\in\Pb$ such that $\val(p_0)=-\infty$ (namely,
problem ($\POPp{p_0}$) does not admit any solution), then by
a well-known result in convex analysis (see, for instance, \cite[Proposition 2.1.4]{Zali02})
it must be $\val(p)=-\infty$ for every $p\in\inte\{p\in\Pb\ |\ \val(p)<+\infty\}=
\inte\Pb=\Pb$.
Otherwise, $\dom\val=\Pb$. In such an event, Theorem \ref{thm:cselectFreg}
ensures the existence of a continuous selection $\sel_\FReg:\Pb\longrightarrow\X$
of $\FReg$. Thus, according to the definition of $\val$ one has
$$
  \val(p)\le\varphi(p,\sel_\FReg(p)),\quad\forall p\in\Pb.
$$
Since $\varphi$ is continuous on $\inte\dom\varphi=\Pb\times\X$,
the composition $p\mapsto\varphi(p,\sel_\FReg(p))$ turns out
to be continuous on $\Pb$. Therefore, it is bounded from above
in a neighbourhood of each point $p\in\Pb$, and so is $\val$
by the above inequality. In the light of \cite[Them 2.2.9]{Zali02}
this suffices to guarantee the local Lipschitz continuity of $\val$ around each point,
and a fortiori its continuity.
\end{proof}

The assumptions of Theorem \ref{thm:lLipval} lead to single-out
a class of parametric constrained optimization problems with
a good marginal behaviour.
Henceforth, any parametric family of optimization problems
such as ($\POP$) satisfying all the assumptions (i)-(vi) in
Theorem \ref{thm:lLipval} will be called {\it qualified convex}
problems {\it with set-valued inclusion constraints} (for short,
{\it q.c.s.v.i.} problems).

\begin{remark}
(i) A first notable consequence of Theorem \ref{thm:lLipval} is the
subdifferentiability of the value function, whenever $\val\ne -\infty$,
associated with any family of q.c.s.v.i. problems.
Indeed, according to \cite[Theorem 2.4.9]{Zali02}, the continuity
and the convexity of $\val$ result in
$$
    \partial\val(p)\ne\varnothing,\quad\forall p\in\Pb.
$$
(ii) Further consequences of the joint convexity and
continuity of $\val$ can be derived when the space $\Pb$ is a Banach
space enjoying special additional properties.
In particular, whenever $\Pb$ is a separable Banach space
(more generally, a weak Asplund space),
then according to \cite[Theorem 5.61]{MorNam22} $\val$ is G\^ateaux
differentiable on a $\Gdelta$ subset of $\Pb$.
Whenever $\Pb$ is a space with a separable dual, then according to
\cite[Theorem 5.66]{MorNam22} $\val$ is even Fr\'echet
differentiable on a $\Gdelta$ subset of $\Pb$.
If, more in particular, $\Pb$ is a finite-dimensional Euclidean space,
then the well-known Rademacher theorem ensures that the points of differentiability
of $\val$ form a full (Lebesgue) measure set.
As remarked in \cite[Chapter 4.6]{Mord18}, the reader should take into account that
the lack of smoothness of $\val$ was one of the major concern in considering
such a fundamental mathematical object in classical calculus of
variations, in contrast stimulating the development
of meaningful constructions in nonsmooth analysis.
\end{remark}

\begin{example}
Let $\Pb=\X=\Y=\R$ and let $C=[0,+\infty)$. Let us consider the family
of problems ($\POP$), which are defined by $\varphi:\R\times\R\longrightarrow\R$
and $F:\R\times\R\rightrightarrows\R$ as follows
$$
\varphi(p,x)=p+x,\qquad F(p,x)=[p-x,+\infty),
$$
respectively. It is then clear that $\FReg:\R\rightrightarrows\R$ is
given by
$$
  \FReg(p)=(-\infty,p),\quad\forall p\in\R,
$$
and consequently
$$
  \val(p)=\inf_{x\in (-\infty,p]}(p+x)=-\infty,\quad\forall
  p\in\R.
$$
Let us check that this family of problems actually falls in the
class q.c.s.v.i..

As $\varphi$ is linear, assumption (i) is trivially fulfilled.

As it is $F(p,x)\backslash [0,+\infty)\subseteq [-|p-x|,0]$ for every
$(p,x)\in\R\times\R$, each set $F(p,x)$ is $[0,+\infty)$-bounded,
which amounts to assumption (ii) being fulfilled.

As for assumptions (iii) and (v), their fulfillment follows at once from the
continuity property of the linear function $h:\R^2\longrightarrow\R$,
being $h(p,x)=p-x$.

Assumption (iv) about the $[0,+\infty)$-concavity of $F$ is satisfied
because it is $F(p,x)=h(p,x)+[0,+\infty)$ and, as a linear mapping, $h$
is $[0,+\infty)$-concave (remember Example \ref{ex:Cconcgomega}).

As for assumption (vi), observe that $F(p,\cdot)$ is $[0,+\infty)$-increasing
at each point $x_0\in\R$, with $\incr{F(p,\cdot)}{x_0}\ge 2$. Indeed, for any
$\delta>0$ and $r\in (0,\delta]$, by taking $u=x_0-r\in\ball{x_0}{r}$, one finds
$$
  \ball{F(p,u)}{2r} = [p-(x_0-r)-2r,+\infty)
  \subseteq [p-x_0-r,+\infty)
   = \ball{F(p,x_0)+[0,+\infty)}{r}.
$$
The above inclusion remains true for every $p\in\Pb$. Thus, for the
set-valued mapping under consideration one obtains
$$
  \alpha_F\ge\inf\{\incr{F(p,\cdot)}{x}\ |\ (p,x)\in\R\times\R\}\ge 2>1,
$$
which shows that the condition in assumption (vi) happens to be satisfied.
\end{example}

The next lemma points out a further property of convexity stemming from
the $C$-concavity for set-valued mapping, which will be employed for
estimating subgradients of $\val$.

\begin{lemma}   \label{lem:convexcF}
Let $C\subseteq\Y$ be a convex cone.
If $F:\Pb\times\X\rightrightarrows\Y$ is $C$-concave then
function $x\mapsto\exc{F(p,x)}{C}$ is convex.
\end{lemma}

\begin{proof}
Observe first that, as $C$ is a convex cone, the function $y\mapsto\dist{y}{C}$
is sublinear. As a consequence, for any $A,\, B\subseteq\Y$ and $t\in (0,+\infty)$,
one has
$$
  \exc{A+B}{C}\le \exc{A}{C}+\exc{B}{C}\quad\hbox{ and }\quad
  \exc{tA}{C}=t\exc{A}{C}.
$$
Moreover, observe that for any $A\subseteq\Y$ it holds $\exc{A+C}{C}=\exc{A}{C}$.
On the account of these observations, taken arbitrary $(p_1,x_1),\, (p_2,x_2)\in
\Pb\times\X$ and $t\in [0,1]$, by exploiting the $C$-concavity of $F$ one can write
\begin{eqnarray*}
  \exc{F(t(p_1,x_1)+(1-t)(p_1,x_1))}{C} &\le &\exc{tF(p_1,x_1)+(1-t)F(p_2,x_2)+C}{C}  \\
   &=& \exc{tF(p_1,x_1)+(1-t)F(p_2,x_2)}{C} \\
   &\le & t\exc{F(p_1,x_1)}{C}+(1-t)\exc{F(p_2,x_2)}{C}.
\end{eqnarray*}
The above inequalities complete the proof.
\end{proof}

The next result establishes an exact formula for calculating the subgradients of $\val$,
which is expressed in terms of problem data ($\varphi$, $F$ and $C$), thereby
providing relevant elements for the sensitivity analysis of ($\POP$).

\begin{theorem}
Let ($\POP$) be a family of q.c.s.v.i. problems, with $\val\ne -\infty$.
Let $\bar p\in\Pb$ and $\bar x\in\Argmin(\bar p)$. If
$\epi\varphi$ and $\graph\FReg\times\R$ are subtransversal at
$(\bar p,\bar x,\varphi(\bar p,\bar x))$, then it holds
\begin{equation}\label{eq:subdval}
  \partial\val(\bar p)=\{p^*+q^*\ |\ (p^*,x^*)\in\partial\varphi(\bar p,\bar x)
  \hbox{ and } (q^*,-x^*)\in\cone\partial\exc{F(\cdot)}{C}(\bar p,\bar x) \}.
\end{equation}
\end{theorem}

\begin{proof}
Under the subtransversality qualification condition it is
possible to employ the exact representation of $\partial\val$
provided by \cite[Theorem 4.5]{HuAnXu24}, which is valid in any
normed space setting.
According to it, one has
$$
  \partial\val(\bar p)=\bigcup_{(p^*,x^*)\in\partial
  \varphi(\bar p,\bar x)}\{p^*+\Coder{\FReg}{\bar p,\bar x}(x^*) \}.
$$
Then, what remains to do is to express the set
$\Coder{\FReg}{\bar p,\bar x}(x^*)$ in terms of the problem data
$F$ and $C$. By recalling the definition of coderivative of a
convex set-valued mapping, this can be done via the normal cone representation
$$
  (q^*,-x^*)\in\Ncone{\graph\FReg}{(\bar p,\bar x)}=\bigcup_{t\ge 0}
   t\partial\dist{\cdot}{\graph\FReg}(\bar p,\bar x),
$$
which has been formulated in Proposition \ref{pro:subdnconerel}(iii).
Indeed, if $(q^*,-x^*)\in\Ncone{\graph\FReg}{(\bar p,\bar x)}$,
then for some $t\ge 0$, as $\bar x\in\FReg(\bar p)$, it must be
\begin{eqnarray*}
  \langle (q^*,-x^*),(p,x)-(\bar p,\bar x)\rangle &=& t\dist{(p,x)}{\graph\FReg} \\
   &\le& t\dist{x}{\FReg(p)},\quad\forall (p,x)\in\Pb\times\X.
\end{eqnarray*}
Thus, since for q.c.s.v.i. problems all the assertions of Theorem \ref{thm:cselectFreg}
hold true, by recalling the global error bound estimate in (\ref{in:glerbosvi}),
if $\alpha\in (1,\alpha_F)$ one obtains
$$
   \langle (q^*,-x^*),(p,x)-(\bar p,\bar x)\rangle\le\frac{t}{\alpha-1}
   \exc{F(p,x)}{C},\quad\forall (p,x)\in\Pb\times\X,
$$
which clearly shows that $(q^*,-x^*)\in\cone\partial\exc{F(\cdot)}{C}
(\bar p,\bar x)$.

On the other hand, it suffices to observe that for every $t\ge 0$
it holds
$$
   t\exc{F(p,x)}{C}\le\ind{(p,x)}{\graph\FReg},\quad\forall (p,x)\in\Pb\times\X,
$$
whence, by passing to the respective subdifferential at $(\bar p,\bar x)$
in both the sides, in the light of (\ref{eq:indsubd}), one finds
$$
  t\partial\exc{F(\cdot)}{C}(\bar p,\bar x)\subseteq\partial
  \ind{\cdot}{\graph\FReg}(\bar p,\bar x)=\Ncone{\graph\FReg}{(\bar p,\bar x)}.
$$
This completes the proof.
\end{proof}

\begin{example}
Consider a parametric class of problems ($\POP$) defined by
$\Pb=\R^s$, $\X=\R^n$, $\Y=\R^m$, a convex function $\varphi:\R^s\times\R^n
\longrightarrow\R$ and by the following constraint system
$$
  H_{\mathcal G}(p,x)\subseteq C,
$$
where $C$ is a nontrivial, closed, pointed, convex cone and
$H_{\mathcal G}$ is a fan compactly generated by the nonempty, convex
compact set ${\mathcal G}\subseteq\Lin(\R^s\times\R^n,\R^m)\cong
\Lin(\R^s,\R^m)\times\Lin(\R^n,\R^m)$.
Notice that if taking an element $(M,\Lambda)\in\Lin(\R^s\times\R^n,\R^m)$,
where $M\in\Lin(\R^s,\R^m)$ and $\Lambda\in\Lin(\R^n,\R^m)$,
its representation matrix is formed by juxtaposing the representation
matrices of $M$ and $\Lambda$, respectively.
As remarked in Example \ref{ex:compgenfan}, $H_{\mathcal G}$ is
a Lipschitz continuous $C$-concave set-valued mapping, taking compact values, so
hypotheses (ii)-(v) of Theorem \ref{thm:lLipval} are satisfied.
Assume that
$$
  \inf_{(M,\Lambda)\in{\mathcal G}}\sur{(M,\Lambda)}=\eta_{\mathcal G}>0
$$
and
$$
   \inte\left(\bigcap_{(M,\Lambda)\in{\mathcal G}}(M,\Lambda)^{-1}(C)
   \right)\ne\varnothing,
$$
so $\incr{H_{\mathcal G}}{p,x}\ge\eta_{\mathcal G}+1$ for every $(p,x)
\in\R^s\times\R^n$. This implies
$$
  \alpha_{H_{\mathcal G}}=\inf\left\{\incr{H_{\mathcal G}(p,\cdot)}{x}\ |\
  (p,x)\in\Pb\times\X,\ H_{\mathcal G}(p,x)\not\subseteq C \right\}\ge
  \eta_{\mathcal G}+1>1.
$$
Thus also hypothesis (vi) of Theorem \ref{thm:lLipval} is fulfilled.
This shows that this kind of problem is actually q.c.s.v.i..
According to Theorem \ref{thm:lLipval}, if $\val\ne-\infty$, then
$\val $ is a convex locally Lipschitz function and, upon the subtransversality
condition on $\epi\varphi$ and $\graph\FReg\times\R$, the following
characterization of its subgradients holds
\begin{equation}    \label{eq:subdvalexc}
 \partial\val(\bar p)=\{p^*+q^*\ |\ (p^*,x^*)\in\partial\varphi(\bar p,\bar x)
  \hbox{ and } (q^*,-x^*)\in\cone\partial\exc{H_{\mathcal G}(\cdot)}{C}
  (\bar p,\bar x) \},
\end{equation}
for every pair $(\bar p,\bar x)$, with $\bar p\in\R^s$ and $\bar x\in\Argmin(\bar p)$.
Denote by
\begin{eqnarray*}
  {\mathcal G}(\bar p,\bar x) = \{(M_0,\Lambda_0)\in{\mathcal G}\ |\
  \exc{H_{\mathcal G}(\bar p,\bar x)}{C} &=& \max_{(M,\Lambda)\in {\mathcal G}}
  \dist{M\bar p+\Lambda\bar x}{C} \\
  &=&\dist{M_0\bar p+\Lambda_0\bar x}{C}\}.
\end{eqnarray*}
Notice that, since function $y\mapsto\dist{y}{C}$ is (Lipschitz) continuous
and the evaluation function $\eva_{(\bar p,\bar x)}:\Lin(\R^s\times\R^n,\R^m)
\longrightarrow\R^m$, i.e. $\eva_{(\bar p,\bar x)}(M,\Lambda)=M\bar p+\Lambda\bar x$,
is continuous, also their composition is continuous. Therefore, by compactness
of ${\mathcal G}$ it must be ${\mathcal G}(\bar p,\bar x)\ne\varnothing$.
According to the subdifferential calculus rule for $\max$ functions,
by taking into account of Proposition \ref{pro:subdnconerel}(i) and
formula (\ref{eq:subdlincomp}), one obtains
\begin{eqnarray*}
  \partial\exc{H_{\mathcal G}(\cdot)}{C}(\bar p,\bar x) &=&\partial
  \max_{(M,\Lambda)\in {\mathcal G}} \dist{M\cdot+\Lambda\,\cdot}{C}(\bar p,\bar x) \\
   &=& \clco\bigcup_{(M,\Lambda)\in {\mathcal G(\bar p,\bar x)}}
   \dist{M\cdot+\Lambda\,\cdot}{C}(\bar p,\bar x) \\
   &=& \clco\bigcup_{(M,\Lambda)\in {\mathcal G(\bar p,\bar x)}}
   (M^*,\Lambda^*)[\Ncone{C}{M\bar p+\Lambda\bar x}\cap\Uball].
\end{eqnarray*}
Thus, from the subgradient representation in (\ref{eq:subdvalexc}) it is
possible to derive the following exact estimate, which is fully expressed
in terms of problem data
\begin{eqnarray}     \label{eq:valsubdcomtpfan}
  \partial\val(\bar p) =  \bigg\{ p^*+q^*  & | &  (p^*,x^*)\in\partial\varphi(\bar p,\bar x)
  \hbox{ and }  \\
   & & (q^*,-x^*)\in\cone\bigl(\clco\bigcup_{(M,\Lambda)\in {\mathcal G(\bar p,\bar x)}}
   (M^*,\Lambda^*)[\Ncone{C}{M\bar p+\Lambda\bar x}\cap\Uball]\bigl) \bigg\} .  \nonumber
\end{eqnarray}
Let us test formula (\ref{eq:valsubdcomtpfan}) in a specific case, where
calculations are easy to be checked.
Let $\Pb=\X=\Y=\R$, $C=[0,+\infty)$, ${\mathcal G}=\{-1\}\times [1,2]$
(with linear mappings being identified with their representation matrix)
and $\varphi:\R\times\R\longrightarrow\R$ given by
$$
  \varphi(p,x)=|p|+x,
$$
so the class ($\POP$) becomes
\begin{equation*}
   \begin{array}{cl}
    \displaystyle\min_{x\in\R}  & |p|+x  \\
      \hbox{ sub } & H_{\{-1\}\times [1,2]}(p,x)=\{-p+\lambda x\ |\
      \lambda\in [1,2]\} \subseteq [0,+\infty).
  \end{array}
\end{equation*}
Since it is
$$
  -p+\lambda x\ge 0,\quad\forall\lambda\in [1,2]
$$
iff it holds
$$
   x\ge \frac{p}{\lambda},\quad\forall\lambda\in [1,2],
$$
the feasible region mapping $\FReg:\R\rightrightarrows\R$ turns out to be
\begin{equation*}
   \FReg(p)=\left\{\begin{array}{ll}
                     \left[p,+\infty)\right. & \hbox{ if }  p\ge 0, \\
                     \\
                     \left[p/2,+\infty)\right. & \hbox{ if }  p< 0.
                 \end{array}\right.
\end{equation*}
Notice that $\graph\FReg$ is a polyhedral convex cone with vertex
at the origin (so $\FReg$ is what is called after Rockafellar a convex
process). Therefore, with the given problem data, one readily deduces that $\Argmin:\R\rightrightarrows\R$
takes the form
\begin{equation*}
   \Argmin(p)=\left\{\begin{array}{ll}
                     \{p\} & \hbox{ if }  p\ge 0, \\
                     \\
                     \{p/2\} & \hbox{ if }  p< 0.
                 \end{array}\right.
\end{equation*}
Consequently, the optimal value function $\val:\R\longrightarrow\R$
associated with the present class of problems results in
\begin{equation*}
   \val(p)=\left\{\begin{array}{ll}
                     2p & \hbox{ if }  p\ge 0, \\
                     \\
                     -p/2 & \hbox{ if }  p< 0.
                 \end{array}\right.
\end{equation*}
As it should be, $\val$ is convex and locally Lipschitz.
By taking into account that, in the present setting, it is
$\Pb^*=\X^*=\R$, one readily sees that
\begin{equation}     \label{eq:subdvalnumex}
  \partial\val(\bar p)=\left\{\begin{array}{ll}
                               \{2\}  & \hbox{ if }  \bar p> 0 \\
                               \\
                                \left[-\displaystyle\frac{1}{2},2\right] & \hbox{ if }  \bar p=0 \\
                                \\
                                \{-\frac{1}{2}\}  & \hbox{ if }  \bar p<0.
                              \end{array}\right.
\end{equation}

Before checking the validity of formula (\ref{eq:valsubdcomtpfan}),
it should be noticed that $\epi\varphi$ is a polyhedral cone (with vertex at the origin)
in $\R^3$ as well as $\graph\FReg\times\R$,
so that the subtransversality condition on $\epi\varphi$ and $\graph\FReg\times\R$
is satisfied at each point $(\bar p,\bar x,\varphi(\bar p,\bar x))$ in as much
both the sets are polyhedral convex sets.
Besides, it is useful to note that, by well-known subdifferential calculus rules,
it holds
\begin{equation*}
  \partial\varphi(p,x)=\left\{\begin{array}{ll}
                               \{(1,1)\}  & \ \forall (p,x)\in (0,+\infty)\times\R   \\
                               \\
                                \left[-1,1\right]\times\{1\} & \ \forall (p,x)\in \{0\}\times\R \\
                                \\
                               \{(-1,1)\}  & \ \forall (p,x)\in (-\infty,0)\times\R.
                              \end{array}\right.
\end{equation*}

\noindent $\bullet$ {\bf Case $\bar p>0$.} In such an event, it must be $\bar x=\bar p$ and hence
\begin{eqnarray*}
   {\mathcal G}(\bar p,\bar p) &=
   \left\{\right. (-1,\lambda_0)\  |&  \ \lambda_0\in [1,2] \hbox{ and }  \\
   &  &\dist{-\bar p+\lambda_0\bar p}{[0,+\infty)}=\max_{\lambda\in [1,2]}
   \dist{-\bar p+\lambda\bar p}{[0,+\infty)} \left.\right\}
      \\
     &= \{-1\}\times [1,2].   &
\end{eqnarray*}
If $(-1,\lambda_0)\in {\mathcal G}(\bar p,\bar p)$ it is
\begin{equation*}
  -1\cdot\bar p+\lambda_0\bar x=(\lambda_0-1)\bar p\ \left\{\begin{array}{ll}
                               \in (0,+\infty)  & \forall\lambda_0\in (1,2]   \\
                               \\
                                =0  &  \hbox{ if }\lambda_0=1.
                              \end{array}\right.
\end{equation*}
Then, one has
\begin{equation*}
  \Ncone{[0,+\infty)}{-\bar p+\lambda_0\bar x}\cap\Uball
  =\left\{\begin{array}{ll}
         \{0\}  & \forall\lambda_0\in (1,2]   \\
                               \\
         \left[-1,0\right] &  \hbox{ if } \lambda_0=1.
         \end{array}\right.
\end{equation*}
This implies
\begin{equation*}
  \{(-t,\lambda_0 t)\ |\ t\in\Ncone{[0,+\infty)}{-\bar p+\lambda_0\bar x}\cap\Uball\}=
  \left\{\begin{array}{ll}
             \{(0,0)\}  & \quad\forall\lambda_0\in (1,2]   \\
                               \\
             \conv\{(0,0),\, (1,-1)\}   & \quad \hbox{ if } \lambda_0=1,
                              \end{array}\right.
\end{equation*}
whence it follows
$$
  \cone\left(\clco\hskip-0.7cm \bigcup_{(-1,\lambda_0)\in{\mathcal G}(\bar p,\bar p)} \hskip-0.7cm
  \{(-t,\lambda_0 t)\ |\ t\in\Ncone{[0,+\infty)}{-\bar p+\lambda_0\bar x}\cap\Uball\}\right) \\
  =\{t(1,-1)\ |\ t\in [0,+\infty)\}.
$$
Since if $(p^*,x^*)\in\partial\varphi(\bar x,\bar p)$ it must be $p^*=1$ and $x^*=1$,
then $(q^*,-1)\in\{t(1,-1)\ |\ t\in [0,+\infty)\}$ only if $t=1$ and hence $q^*=1$.
Thus, according to formula (\ref{eq:valsubdcomtpfan}) one obtains
$$
  \partial\val(\bar p)=\{p^*+q^*\}=\{2\},
$$
consistently with the value in (\ref{eq:subdvalnumex}).

\vskip.5cm

\noindent $\bullet$ {\bf Case $\bar p=0$.} In such an event, it must be $\bar x=\bar p=0$ and hence
\begin{eqnarray*}
   {\mathcal G}(0,0) &=
   \left\{\right. (-1,\lambda_0)\  |&  \ \lambda_0\in [1,2] \hbox{ and }  \\
   &  &\dist{0}{[0,+\infty)}=\max_{\lambda\in [1,2]}
   \dist{-1\cdot 0+\lambda\cdot 0}{[0,+\infty)} \left.\right\}
      \\
     &= \{-1\}\times [1,2].   &
\end{eqnarray*}
If $(-1,\lambda_0)\in {\mathcal G}(0,0)$ it is $-1\cdot 0+\lambda_0\cdot 0=0$
which yields
$$
  \Ncone{[0,+\infty)}{0}\cap\Uball=[-1,0],\quad\forall\lambda_0\in [1,2].
$$
This implies
\begin{eqnarray*}
  \{(-t,\lambda_0 t)\ |\ t\in\Ncone{[0,+\infty)}{0}\cap\Uball\} &=&
  \bigcup_{\lambda_0\in [1,2]} \conv\{(0,0),\, (1,-\lambda_0)\}   \\
  &=& \conv\{(0,0),\, (1,-1),\, (1,-2)\},
\end{eqnarray*}
whence it follows
$$
  \cone\left(\clco\hskip-0.7cm \bigcup_{(-1,\lambda_0)\in{\mathcal G}(0,0)} \hskip-0.7cm
  \{(-t,\lambda_0 t)\ |\ t\in\Ncone{[0,+\infty)}{0}\cap\Uball\}\right) \\
  =\cone\left(\conv\{(0,0),\, (1,-1),\, (1,-2)\}\right).
$$
Since if $(p^*,x^*)\in\partial\varphi(\bar x,\bar p)$ it is $p^*\in [-1,1]$ and $x^*=1$,
then $(q^*,-1)\in\cone\left(\conv\{(0,0),\, (1,-1),\, (1,-2)\}\right)$ only if  $q^*\in
\left[\frac{1}{2},1\right]$.
Thus, according to formula (\ref{eq:valsubdcomtpfan}) one obtains
$$
  \partial\val(\bar p)=[-1,1]+\left[\frac{1}{2},1\right]=
  \left[-\frac{1}{2},2\right],
$$
which is consistent with the set found in (\ref{eq:subdvalnumex}).

\vskip.5cm

\noindent $\bullet$ {\bf Case $\bar p<0$.} In such an event, it must be
$\bar x=\bar p/2$, and hence
\begin{eqnarray*}
   {\mathcal G}(\bar p,\bar p/2) &=
   \biggl\{ (-1,\lambda_0)\  |&  \ \lambda_0\in [1,2] \hbox{ and }  \\
   &  &\dist{-\bar p+\lambda_0\frac{\bar p}{2}}{[0,+\infty)}=\max_{\lambda\in [1,2]}
   \dist{-\bar p+\lambda\frac{\bar p}{2}}{[0,+\infty)} \biggr\}
      \\
     &= \{-1\}\times [1,2].   &
\end{eqnarray*}
If $(-1,\lambda_0)\in {\mathcal G}(\bar p,\bar p/2)$ it is
\begin{equation*}
  -1\cdot\bar p+\lambda_0\bar x=\left(\frac{\lambda_0}{2}-1\right)\bar p\ \left\{\begin{array}{ll}
                                  \in (0,+\infty)  & \forall\lambda_0\in [1,2)   \\
                               \\
                                =0  &  \hbox{ if }\lambda_0=2.
                              \end{array}\right.
\end{equation*}
Then, one has
\begin{equation*}
  \Ncone{[0,+\infty)}{-\bar p+\lambda_0\bar x}\cap\Uball
  =\left\{\begin{array}{ll}
         \{0\}  & \forall\lambda_0\in [1,2)   \\
                               \\
         \left[-1,0\right] &  \hbox{ if }\lambda_0=2.
         \end{array}\right.
\end{equation*}
This implies
\begin{equation*}
  \{(-t,\lambda_0 t)\ |\ t\in\Ncone{[0,+\infty)}{-\bar p+\lambda_0\bar x}\cap\Uball\}=
  \left\{\begin{array}{ll}
             \{(0,0)\}  & \quad\forall\lambda_0\in [1,2)   \\
                               \\
             \conv\{(0,0),\, (1,-2)\}   & \quad \hbox{ if } \lambda_0=2,
                              \end{array}\right.
\end{equation*}
which gives
$$
  \cone\left(\clco\hskip-0.7cm \bigcup_{(-1,\lambda_0)\in{\mathcal G}(\bar p,\bar p/2)} \hskip-0.7cm
  \{(-t,\lambda_0 t)\ |\ t\in\Ncone{[0,+\infty)}{-\bar p+\lambda_0\bar x}\cap\Uball\}\right) \\
  =\{t(1,-2)\ |\ t\in [0,+\infty)\}.
$$
It is clear that  if $(p^*,x^*)\in\partial\varphi(\bar x,\bar p/2)$ it must be $p^*=-1$ and $x^*=1$,
then $(q^*,-1)\in\{t(1,-2)\ |\ t\in [0,+\infty)\}$ only if $t=\frac{1}{2}$ and hence $q^*=\frac{1}{2}$.
Thus, according to formula (\ref{eq:valsubdcomtpfan}), one obtains
$$
  \partial\val(\bar p)=\{p^*+q^*\}=\left\{-1+\frac{1}{2}\right\}=
  \left\{-\frac{1}{2}\right\},
$$
again consistently with the value in (\ref{eq:subdvalnumex}).
\end{example}

\vskip1cm


\section{Problem calmness}      \label{Sect:4}

The stability and sensitivity analysis conducted in the previous section
is complemented here with a related issue, dealing with a certain stability
behaviour that parametric constrained optimization problems may exhibit.
Proposed by R.T. Rockafellar, such behaviour was studied in \cite{Clar76}
and then employed in connection with penalization reduction
procedures within perturbed nonlinear programming in \cite{Burk91}
and subsequent works.

\begin{definition}[Problem calmness]     \label{def:propcalm}
Given a family of problems ($\POP$),
let $\bar p\in\Pb$ and let $\bar x\in\Argmin(\bar p)$.
Problem ($\POPp{\bar p}$) is said to be {\it calm} at $\bar x$ if $\exists r,\,\ \lambda>0$:
\begin{equation}    \label{in:defprobcalm}
  \inf_{p\in\ball{\bar p}{r}\backslash\{\bar p\}}\ \
  \inf_{x\in\ball{\bar x}{r}\cap\FReg(p)}
  \frac{\varphi(p,x)-\varphi(\bar p,\bar x)}{\|p-\bar p\|}\ge -\lambda.
\end{equation}
\end{definition}

As a straightforward consequence of Theorem \ref{thm:lLipval}, for the
class of parametric optimization problems under investigations problem
calmness comes under natural qualified convexity assumptions.

\begin{corollary}    \label{cor:qcsvicalm}
If the problems in ($\POP$) are q.c.s.v.i., then for any $\bar p\in\dom\val=\Pb$,
problem ($\POPp{\bar p}$) is calm at any $\bar x\in\Argmin(\bar p)$.
\end{corollary}

\begin{proof}
According to Theorem \ref{thm:lLipval}, there exist positive
$r$ and $\ell$ such that
$$
  |\val(p_1)-\val(p_2)|\le\ell\|p_1-p_2\|,\quad\forall
  p_1,\, p_2\in\ball{\bar p}{r}.
$$
On the account of this inequality it follows
$$
   \inf_{p\in\ball{\bar p}{r}\backslash\{\bar p\}}\ \
  \inf_{x\in\ball{\bar x}{r}\cap\FReg(p)}
  \frac{\varphi(p,x)-\varphi(\bar p,\bar x)}{\|p-\bar p\|}\ge
  \inf_{p\in\ball{\bar p}{r}\backslash\{\bar p\}}\ \
  \frac{\val(p)-\val(\bar p)}{\|p-\bar p\|}\ge -\ell.
$$
\end{proof}

\begin{definition}     \label{def:penalf}
Let $\bar x\in F^{+1}(C)$ be a local solution of the problem
\begin{equation*}
  (\OP) \hskip 1.8cm \begin{array}{cl}
    \displaystyle\min_{x\in X}  & \varphi(x) \\
      \hbox{ sub } & F(x)\subseteq C.
  \end{array}
\end{equation*}
with a set-valued inclusion constraint.
$(\OP)$ is said to admit a {\it penalty function} at $\bar x$
if there exists $\lambda_*\ge 0$ such that for every $\lambda\in(\lambda_*,+\infty)$
$\bar x$ is an unconstrained local minimizer of
\begin{equation*}
  \varphi_\lambda(x)=\varphi(x)+\lambda\exc{F(x)}{C}.
\end{equation*}
\end{definition}

In the next result, a sufficient condition for the existence of a penalty function
for general parametric optimization problems with set-valued inclusion constraints
is formulated, where problem calmness plays a crucial role.

\begin{theorem}     \label{thm:penalfexist}
With reference to the family of problems ($\POP$), let $\bar p\in\Pb$
and let $\bar x\in\Argmin(\bar p)$.
Suppose that:
\begin{itemize}
\item[(i)] $\varphi(\bar p,\cdot)$ is l.s.c. at $\bar x$;

\item[(ii)] $\varphi(\cdot,\bar x)$ is calm from above at $\bar p$,
  uniformly in $x$, i.e. there exist positive $\gamma$ and $r_\gamma$ such that
$$
   \varphi(p,x)-\varphi(\bar p,x)\le\gamma\|p-\bar p\|,\quad
   \forall p\in\ball{\bar p}{r_\gamma},\ \forall x\in\ball{\bar x}{r_\gamma};
$$

\item[(iii)] there exist positive $\beta$ and $r_\beta$ such that
\begin{equation}   \label{in:methem}
  \dist{\bar p}{F^{+1}(\cdot,x)(C)}\le\beta\exc{F(\bar p,x)}{C},
  \quad\forall x\in\ball{\bar x}{r_\beta};
\end{equation}

  \item[(iv)] ($\POPp{\bar p}$) is calm at $\bar x$.
\end{itemize}
Then, ($\POPp{\bar p}$) admits a penalty function at $\bar x$.
\end{theorem}

\begin{proof}
Ab absurdo, assume that ($\POPp{\bar p}$) fails to admit a penalty function at
$\bar x$. This means that for each $\lambda_*\ge 0$ there exist
$k\in\N$, with $k>\lambda_*$, and $x_k\in\ball{\bar x}{1/k}$, such that
\begin{equation}\label{in:nopenf}
  \varphi(\bar p,x_k)+k\exc{F(\bar p,x_k)}{C}<\varphi(\bar p,\bar x).
\end{equation}
Since $\bar x$ is a local solution to ($\POPp{\bar p}$) the inequality in
(\ref{in:nopenf}) implies the existence of $k_0\in\N$ such that, for every
$k\in\N$, with $k\ge k_0$, it is $x_k\not\in\FReg(\bar p)$, so
$$
  \exc{F(\bar p,x_k)}{C}>0.
$$
Because of $x_k\longrightarrow\bar x$ as $k\to\infty$, then by taking into account
hypothesis (i), from (\ref{in:nopenf}) one deduces
$$
  \limsup_{k\to\infty}k\exc{F(\bar p,x_k)}{C}\le\limsup_{k\to\infty}
  [\varphi(\bar p,\bar x)-\varphi(\bar p,x_k)]=
  \varphi(\bar p,\bar x)-\liminf_{k\to\infty} \varphi(\bar p,x_k)\le 0.
$$
The last inequality entails
\begin{equation}\label{eq:limexc0}
  \exists\ \lim_{k\to\infty}\exc{F(\bar p,x_k)}{C}=0^+.
\end{equation}
On the other hand, again because of $x_k\longrightarrow\bar x$ as $k\to\infty$,
up to an increase of the value of $k_0$ if needed, one has
$x_k\in\ball{\bar x}{r_\beta}$, where $r_\beta>0$ is as in hypothesis (iii).
Consequently, one obtains
$$
  \dist{\bar p}{F^{+1}(\cdot,x_k)(C)}\le\beta\exc{F(\bar p,x_k)}{C},
  \quad\forall k\in\N,\ k\ge k_0.
$$
Thus, by taking $\tilde{\beta}>\beta$ and recalling that $\dist{\bar p}{F^{+1}(\cdot,x_k)(C)}>0$
as $x_k\not\in\FReg(\bar p)$, for each $k\ge k_0$ it is possible
to get the existence of $p_k\in F^{+1}(\cdot,x_k)(C)$ with the property that
\begin{equation}    \label{in:tildebetaexc}
  \tilde{\beta}^{-1}d(\bar p,p_k)<\exc{F(\bar p,x_k)}{C}
\end{equation}
and $F(p_k,x_k)\subseteq C$, so $x_k\in\FReg(p_k)$.
Notice that it must be $p_k\ne\bar p$ for every $k\ge k_0$,
otherwise it would result in $x_k\in\FReg(\bar p)$, what has been already
excluded above. Moreover, on the account of (\ref{eq:limexc0}), the estimate
in (\ref{in:tildebetaexc}) entails that the sequence $(p_k)_k$ converges
to $\bar p$ as $k\to\infty$.
By combining the inequalities in (\ref{in:nopenf}) and in
(\ref{in:tildebetaexc}), one obtains
$$
  \frac{\varphi(\bar p,x_k)-\varphi(\bar p,\bar x)}{\tilde{\beta}^{-1}d(\bar p,p_k)}
  \le \frac{\varphi(\bar p,x_k)-\varphi(\bar p,\bar x)}{\exc{F(\bar p,x_k)}{C}}<-k,
  \quad\forall k\ge k_0,
$$
which yields
\begin{equation}     \label{in:phidkbeta}
  \frac{\varphi(\bar p,x_k)-\varphi(\bar p,\bar x)}{d(\bar p,p_k)}
  \le -\frac{k}{\tilde{\beta}}, \quad\forall k\ge k_0.
\end{equation}
Since by hypothesis (ii) and the convergence of $(p_k)_k$, up to a further increase of the value
of $k_0$, if needed, so that $x_k\in\ball{\bar x}{r_\gamma}$ and $p_k\in\ball{\bar p}{r_\gamma}$,
for $\gamma>0$ one has
$$
  \varphi(p_k,x_k)-\gamma\|p_k-\bar p\|\le\varphi(\bar p,x_k),
  \quad\forall k\ge k_0,
$$
from the inequality in (\ref{in:phidkbeta}) it follows
$$
  \frac{\varphi(p_k,x_k)-\varphi(\bar p,\bar x)}{\|\bar p-p_k\|}
  \le -\frac{k}{\tilde{\beta}}+\gamma, \quad\forall k\ge k_0.
$$
This amounts to say that for every $k\ge k_0$ there exists $p_k\in\ball{\bar p}{r_\gamma}
\backslash\{\bar p\}$ such that
$$
  \inf_{x\in\ball{\bar x}{r_\gamma}\cap\FReg(p_k)}
  \frac{\varphi(p_k,x)-\varphi(\bar p,\bar x)}{\|\bar p-p_k\|}
  \le -\frac{k}{\tilde{\beta}}+\ell.
$$
The last inequality shows that the condition (\ref{in:defprobcalm}) in
Definition \ref{def:propcalm} is violated, what contradicts the assumption (iv)
on the calmness of ($\POPp{\bar p}$) at $\bar x$.
This completes the proof.
\end{proof}

\begin{remark}    \label{rem:philocLipmsubreg}
(i) It is readily seen that both hypotheses (i) and (ii) of
Theorem \ref{thm:penalfexist} happen to be satisfied, in particular,
by any function $\varphi$, which is locally Lipschitz around $(\bar p,\bar x)$.

(ii) Hypothesis (iii) can be regarded as an error bound condition leading
to a kind of metric subregularity for the set-valued mapping $F(\cdot,x)$
at $\bar p$, which is uniform with respect to $x$ around $\bar x$.
\end{remark}

\begin{corollary}
If the problems in ($\POP$) are q.c.s.v.i. and for any $\bar p\in\dom\val=\Pb$
there exist positive $\beta$ and $r_\beta$ for which the inequality in (\ref{in:methem})
is true, then problem ($\POPp{\bar p}$) admits a penalty function at any $\bar x\in\Argmin(\bar p)$.
\end{corollary}

\begin{proof}
As a continuous convex function, $\varphi$ is also locally Lipschitz around
$(\bar p,\bar x)$. On the basis of what was observed in Remark \ref{rem:philocLipmsubreg}(i),
both hypotheses (i) and (ii) of Theorem \ref{thm:penalfexist} are then fulbilled.
Since problems in in ($\POP$) are q.c.s.v.i., then by Corollary \ref{cor:qcsvicalm}
($\POPp{\bar p}$) is a calm problem at any $\bar x\in\Argmin(\bar p)$.
Thus, the thesis follows from Theorem \ref{thm:penalfexist}.
\end{proof}

It is clear that the existence of a penalty function paves the way
to deriving optimality conditions for constrained optimization
problems from those valid in the unconstrained case. In the
specific case of q.c.s.v.i. problems, it must be noticed that,
since $\FReg$ is a convex multifunction, so it takes convex values,
each problem in ($\POP$) falls in the realm of convex optimization.
Consequently, any solution $\bar x\in\Argmin(\bar p)$ is global.
Moreover, in the light of Lemma \ref{lem:convexcF}, each penalized
function $x\mapsto\varphi_\lambda(\bar p,x)=\varphi(\bar p,x)+\lambda
\exc{F(\bar p,x)}{C}$ is convex, so many theoretical and
computational tools are at disposal for its minimization.
In particular, on the basis of the specific form taken
by $F$, various constructions of generalized differentiation
and subdifferential calculus rules could be exploited in order
to formulate first-order optimality conditions for the unconstrained
minimization problem
$$
  \min_{x\in\X}[\varphi(\bar p,x)+\lambda\exc{F(\bar p,x)}{C}].
$$
The exploration and assessment of such kind of research perspectives
will be the theme of future investigations.



\vskip 6mm


\end{document}